\documentclass[12pt]{amsart}

\usepackage{graphicx}
\usepackage{amsmath}
\usepackage{amsfonts}
\usepackage{amssymb}
\usepackage{pstricks}
\usepackage{url}
\usepackage{nicefrac}
\usepackage{color}
\usepackage{subfigure}
\input xy
\xyoption{all}

\newtheorem{theorem}{Theorem}

\newtheorem{corollary}[theorem]{Corollary}

\newtheorem{lemma}[theorem]{Lemma}

\newtheorem{proposition}[theorem]{Proposition}

\theoremstyle{definition}
\newtheorem{remark}[theorem]{Remark}

\renewcommand{\ker}{\mathrm{ker}}
\newcommand{\im}{\mathrm{im}}

\newcommand{\Per}{\mathrm{Per}}

\newcommand{\id}{\mathrm{Id}}
\newcommand{\diam}{\mathrm{diam}}
\newcommand{\Hom}{\mathrm{Hom}}

\title{Dynamics and eigenvalues in dimension zero}

\author[Hern\'andez-Corbato]{Luis Hern\'andez-Corbato}
\author[Nieves-Rivera]{David Jes\'us Nieves-Rivera}
\author[Ruiz del Portal]{Francisco R. Ruiz del Portal}
\author[S\'anchez-Gabites]{Jaime J. S\'anchez-Gabites}


\thanks{The authors have been supported by MINECO grant MTM2015-63612-P. The second author is a beneficiary of a predoctoral contract financed by the Santander-UCM predoctoral aid program.}

\address{L. Hern\'andez Corbato \\ Departamento de Matem\'atica Aplicada a las TIC \\ Universidad Polit\'ecnica de Madrid \\  28031 Madrid \\ Spain}
\email{luis.hcorbato@upm.es}

\address{D. J. Nieves-Rivera \\ Departamento de \'Algebra, Geometr\'{\i}a y Topolog\'{\i}a\\ Universidad Complutense de Madrid \\ Plaza de Ciencias 3 \\ 28040 Madrid \\ Spain}
\email{davidjni@ucm.es}

\address{F. R. Ruiz del Portal \\ Departamento de \'Algebra, Geometr\'{\i}a y Topolog\'{\i}a\\ Universidad Complutense de Madrid \\ Plaza de Ciencias 3 \\ 28040 Madrid \\ Spain}
\email{rrportal@ucm.es}

\address{J. J. S{\'{a}}nchez-Gabites \\ Departamento de An\'alisis Econ\'omico (M\'etodos cuantitativos) \\ Facultad de Ciencias Econ\'omicas y Empresariales \\ Universidad Aut\'onoma de Madrid \\ 28049 Madrid \\ Spain}
\email{JaimeJ.Sanchez@uam.es}

\begin{document}

\begin{abstract}
Let \(X\) be a compact, metric and totally disconnected space and let \(f \colon X \rightarrow X\) be a continuous map. We relate the eigenvalues of \(f_*\colon \Check{H}_0(X; {\mathbb C}) \rightarrow \Check{H}_0(X;{\mathbb C})\) to dynamical properties of \(f\),  roughly showing that if the dynamics is complicated then every complex number of modulus different from 0, 1 is an eigenvalue. This stands in contrast with the classical Manning's inequality.
\end{abstract}

\maketitle

\section{Introduction}

Manning's theorem \cite{manning} essentially says that 
if $X$ is a compact manifold (or, more generally, an absolute neighborhood retract for metric spaces) and $f \colon X \rightarrow X$ is a continuous
map, then
\[
h(f) \geq \log(|\lambda|)
\]

\noindent for any eigenvalue $\lambda$ of $f_*\colon H_1(X; \mathbb{C})
\rightarrow H_1(X; \mathbb{C})$. Here \(h(f)\) stands for the topological entropy of \(f\) and \(H\) denotes singular homology with coefficients in $\mathbb{C}$. Motivated by this result, Shub stated the Entropy Conjecture which asks if Manning's theorem is true in all dimensions. In this paper we consider this problem in dimension zero, obtaining results that relate the eigenvalues of $f_*$ to certain detailed dynamical properties of the system. 

We will concentrate on $X$ compact and totally disconnected (and metrizable). The reason is the following. Since $0$--dimensional homological invariants measure, roughly, only connectedness properties, it is to be expected that the results that we obtain do not involve the dynamical behaviour \emph{within} the connected components of our phase space, but rather \emph{among} them. Thus we may as well collapse each component to a singleton, obtaining a compact, totally disconnected space, and consider the induced map in this quotient space.

The interesting cases arise when $X$ is infinite and has a very complicated topological structure, which is a common situtation in both discrete and continuous dynamical systems. The homology or cohomology theory that best suits the study of these spaces is not singular homology as in Manning's theorem, but rather \v{C}ech homology or even \v{C}ech cohomology. In dimension zero these have a relatively simple description. For an arbitrary compact space $Z$, its \v{C}ech cohomology group $\check{H}^0(Z;\mathbb{C})$ can be identified with the set of all locally constant maps $\varphi \colon Z \to \mathbb{C}$. Because coefficients are taken in $\mathbb{C}$, in this case $\check{H}^0(Z;\mathbb{C})$ is actually a vector space over $\mathbb{C}$. \v{C}ech homology, on the other hand, can be shown to be the dual vector space to \v{C}ech cohomology; that is, $\check{H}_0(Z;\mathbb{C}) = \Hom (\check{H}^0(Z;\mathbb{C}),\mathbb{C})$. 


The result for \v{C}ech cohomology is very simple to state.

\begin{proposition} \label{prop:coh} Let $X$ be compact, metric and totally disconnected space and let $f \colon X \to X$ be continuous. Consider the induced map $f^* \colon \check{H}^0(X;\mathbb{C}) \to \check{H}^0(X;\mathbb{C})$. Then every nonzero eigenvalue of $f^*$ is a root of unity.
\end{proposition}

In particular, the inequality of Manning's theorem holds trivially since it just says that $h(f) \geq 0$. The scenario changes radically, however, when one considers \v{C}ech homology:

\begin{theorem} \label{thm:intro} Let $X$ be compact, metric and totally disconnected space and let $f \colon X \to X$ be continuous. Consider the induced map $f_* \colon \check{H}_0(X;\mathbb{C}) \to \check{H}_0(X;\mathbb{C})$. Then the following statements are equivalent:
\begin{itemize}
	\item[({\it i}\/)] There exists $\lambda \in \mathbb{C}$ with $|\lambda| \neq 0, 1$ that is not an eigenvalue of $f_*$.
	\item[({\it ii}\/)] No $\lambda \in \mathbb{C}$ with $|\lambda| \neq 0, 1$ is an eigenvalue of $f_*$.
	\item[({\it iii}\/)] $(X, f)$ admits dynamical $\epsilon$--partitions for every $\epsilon > 0$.
\end{itemize}
\end{theorem}

A dynamical $\epsilon$--partition $\mathcal U$ is a finite partition of $X$ into clopen (that is, simultaneously closed and open) subsets of diameter smaller than $\epsilon$ and such that the image of any element of $\mathcal U$ is completely contained in another (not necessarily different) element of $\mathcal U$.

The existence of a dynamical $\epsilon$--partition $\mathcal U$ implies that if $x, y$ belong to the same element of $\mathcal{U}$ then their forward images never separate a distance greater than $\epsilon$. Thus, if $\delta$ is a Lebesgue number for $\mathcal{U}$ (for instance, if $\delta$ is smaller than the minimum distance among elements of $\mathcal U$), it follows that for every pair of points such that $d(x, y) < \delta$ then $d(f^n(x), f^n(y)) < \epsilon$ for every $n \ge 1$. That is, $f$ is not positively expansive unless $X$ is finite. This argument also shows that the number of $\epsilon$--distinguishable positive semiorbits is finite and, in particular, the entropy of $f$ is zero. Thus, we have the following corollary of Theorem \ref{thm:intro}:

\begin{corollary}\label{cor:entropy}
Let $X$ be a compact, metric and totally disconnected space and $f \colon X \to X$ continuous. If the topological entropy of $f$ is nonzero or $f$ is positively expansive and $X$ is infinite, then every $\lambda \in \mathbb C$ with $|\lambda| \neq 0, 1$ is an eigenvalue of $f_* \colon \check{H}_0(X;\mathbb{C}) \to \check{H}_0(X;\mathbb{C})$.
\end{corollary}

In particular, the inequality of Manning's theorem certainly does not hold in general: for any dynamical system with positive but finite topological entropy we have $\sup\ \log(|\lambda|) = +\infty$ as $\lambda$ ranges over the eigenvalues of $f_*$.

Actually, the existence of dynamical $\epsilon$--partitions for every $\epsilon > 0$ is a very stringent condition; so much so that it leads to a fairly detailed description of the dynamical system $(X,f)$:

\begin{theorem} \label{thm:intro2} Let $X$ be compact, metric and totally disconnected space and let $f \colon X \to X$ be continuous. The following statements are equivalent:
\begin{itemize}
	\item[({\it i}\/)] $(X, f)$ admits dynamical $\epsilon$--partitions for every $\epsilon > 0$.
	\item[({\it ii}\/)] The $\omega$--limit of every point in $X$ is either a periodic orbit or an adding machine; moreover, these are stable with respect to clopen sets.
\end{itemize}
\end{theorem}

Recall that a closed invariant subset $Y \subset X$ is called Lyapunov stable if it has a basis of positively invariant neighbourhoods. We introduce \emph{stability with respect to clopen sets} as an even stronger condition: it means that $Y$ has a basis of \emph{clopen} positively invariant neighbourhoods. This condition will arise naturally as we prove Theorem \ref{thm:intro}.

The characterization of Theorem \ref{thm:intro2} is closely related to work by Buescu, Kulczyki and Stewart \cite{buescu2}, \cite{buescu1}.


The paper is organized as follows. Section \ref{sec:partitions} presents elementary definitions and results about partitions. Section \ref{sec:coh} contains the proof of Proposition \ref{prop:coh}. Theorem \ref{thm:intro} will be proved in Section \ref{sec:eigen}. Theorem \ref{thm:intro2} will be proved in Section \ref{sec:characterization}. We thank the referee for suggesting an approach that allowed us to shorten the proofs in this section. For completeness we have included a brief appendix that contains basic definitions and results about \v{C}ech homology and cohomology.


\section{Elementary definitions and results about partitions} \label{sec:partitions}

Here we gather some definitions and simple results about partitions that we will use frequently throughout the paper. As usual, $X$ denotes a compact, totally disconnected metric space and $f \colon X \to X$ is a continuous map. We recall that such a space $X$ has a topological basis of clopen sets. We will also tacitly assume $X$ to be metrizable so that we can measure the smallness of our partitions with the numerical parameter $\epsilon$. (This assumption can be easily dispensed with at the cost of the additional burden of having to speak of cofinal families of partitions within the family of all open coverings of $X$.)

The pair $(X, f)$ is a (semi)-dynamical system. An action map $\pi$ between dynamical systems $(X, f)$ and $(Y, g)$ is a continuous map $\pi \colon X \to Y$ such that $g \circ \pi = \pi \circ f$, that is, $\pi$ takes $f$--orbits onto $g$--orbits. A surjective action map is typically called a semiconjugacy and if the action map is a homeomorphism then it is a conjugacy between dynamical systems.

A partition of $X$ is a collection $\mathcal{U}$ of pairwise disjoint subsets of $X$ whose union is equal to $X$. We remark once and for all that \emph{the partitions in this work will always be finite and consist of clopen sets}, although for the sake of brevity we will not state this explicitly in the sequel. Since $X$ is compact, any such partition is finite. We shall say that $\mathcal U$ is an \emph{$\epsilon$--partition} if all its elements have diameter smaller than $\epsilon$. In particular, letting $\delta$ be the minimum distance between pairs of elements of $\mathcal{U}$ gives a Lebesgue number for the partition; that is, any set of diameter less than $\delta$ is contained in precisely one member of $\mathcal{U}$. As a consequence, if $\mathcal{V}$ is a $\delta'$--partition and $\delta' < \delta$ then $\mathcal{V}$ refines $\mathcal{U}$.

A partition $\mathcal U$ is called \emph{dynamical} if for every $U \in \mathcal U$ there exists $U' \in \mathcal{U}$ such that $f(U) \subset U'$. Notice that then $f$ induces a map $\tau : \mathcal{U} \to \mathcal{U}$ defined by $\tau(U) := U'$. If $f$ is surjective clearly this map $\tau$ is surjective (hence a bijection, since $\mathcal{U}$ is finite). The carrier map $i : X \to \mathcal{U}$ associates to every $p \in X$ the member $U$ of $\mathcal{U}$ to which $p$ belongs. This carrier map is an action map (in fact, it is a semiconjugacy) between the dynamics given by $f$ on $X$ and that given by $\tau$ on $\mathcal{U}$; that is, $\tau \circ i = i \circ f$. Observe that $i$ takes the $f$--orbit of a point $p \in X$ onto the $\tau$--orbit of $i(p)$ in $\mathcal{U}$. For notational convenience, on occasion we will index the elements of $\mathcal{U}$ with some index set $\Lambda$ and then the map $\tau$ can equivalently be thought of as a map $\tau \colon \Lambda \to \Lambda$.



 Given two clopen partitions $\mathcal{U}$ and $\mathcal{U}'$, their common refinement $\mathcal{U} \vee \mathcal{U}'$ is defined as the partition whose elements are $U \cap U'$ where $U \in \mathcal{U}$ and $U' \in \mathcal{U}'$ (of course, many of these sets will usually be empty). Clearly $\mathcal{U} \vee \mathcal{U}'$ is also a clopen partition. The following are easy to prove:
\begin{itemize}
	\item If $\mathcal{U}$ and $\mathcal{U}'$ are dynamical partitions, then so is $\mathcal{U} \vee \mathcal{U}'$.
	\item If $\mathcal{U}$ or $\mathcal{U}'$ is an $\epsilon$--partition, then so is $\mathcal{U} \vee \mathcal{U}'$.
	\item More generally, suppose that $\mathcal{U}$ and $\mathcal{U}'$ are such that for every $p \in X$ the element in either $\mathcal{U}$ or $\mathcal{U}'$ (or both) that contains $p$ has diameter smaller than $\epsilon$. Then again $\mathcal{U} \vee \mathcal{U}'$ is an $\epsilon$--partition.
\end{itemize}

Let $\mathcal{U} = \{U_0,\ldots,U_n\}$ be a partition of $X$. The itinerary $I$ of a point $p \in X$ with respect to the partition $\mathcal{U}$ is the sequence $I = a_0 a_1 \ldots$ where each $a_k$ is the label $i$ of the set $U_i$ to which $f^k(p)$ belongs; that is, if $f^k(p) \in U_i$ we set $a_k = i$. Notice that the itinerary of $f(p)$ is then $a_1 a_2 \ldots$, which can be thought of as the result of deleting the first symbol in $I$ and shifting everything one position to the left, so that $a_1$ now occupies the $0$th position and so on. This is nothing but the well known shift map acting on $I$. We shall denote this map by $\sigma$.

We will often write ``there are only finitely many itineraries with respect to $\mathcal{U}$'' to mean that the set $\mathcal{I}$ of itineraries with respect to $\mathcal{U}$ that are \emph{actually realized} by points in $X$ is finite. If $p$ realizes an itinerary $I$ then $f(p)$ realizes $\sigma(I)$; thus, $\sigma$ maps $\mathcal{I}$ into $\mathcal{I}$. The following two assertions should be clear:
\begin{itemize}
	\item If $\mathcal{U}$ and $\mathcal{U}'$ are two partitions such that there are only finitely many itineraries with respect to each of them, the same is true of $\mathcal{U} \vee \mathcal{U}'$.
	\item If $\mathcal{U}$ is a dynamical partition, then there are only finitely many itineraries with respect to $\mathcal{U}$ (this uses our tacit convention that all partitions that we consider are finite).
\end{itemize}

Suppose that there are only finitely many itineraries with respect to a partition $\mathcal{U}$ and let $\mathcal{I}$ be as before the set of those itineraries. Consider the shift map $\sigma \colon \mathcal{I} \to \mathcal{I}$ and the nested sequence of images $\mathcal{I} \supset \sigma(\mathcal{I}) \supset \sigma^2(\mathcal{I}) \supset \ldots$ Since $\mathcal{I}$ is finite by assumption, there exists $k$ such that $\sigma^k(\mathcal{I}) = \sigma^{k+1}(\mathcal{I}) = \ldots$ Let us call $\mathcal{I}_0$ this set onto which the images of $\sigma$ stabilize. Then the restriction $\sigma|_{\mathcal{I}_0} \colon \mathcal{I}_0 \to \mathcal{I}_0$ is surjective and, since $\mathcal{I}_0$ is finite, a bijection. In particular there exists $s$ such that $(\sigma|_{\mathcal{I}_0})^s = {\rm id}$. This amounts to saying that every itinerary in $\mathcal{I}_0$ is periodic of (not necessarily minimal) period $s$. Thus we have proved the following:

\begin{itemize}
	\item If there are only finitely many itineraries with respect to $\mathcal{U}$, there exist $k$ and $s$ such that for any $p \in X$ the itinerary of $f^k(p)$ is $s$--periodic.
\end{itemize}

Notice that the set of itineraries of points in $f(X)$ is precisely $\sigma(\mathcal{I})$. If $f$ is surjective then $f(X) = X$ and so $\sigma(\mathcal{I}) = \mathcal{I}$ must hold. Then we may take $k = 0$ in the above arguments and conclude that every itinerary is $s$--periodic. In particular $f^s(U_i) \subset U_i$ and, because the $U_i$ partition $X$ and $f$ is surjective, we must actually have $f^s(U_i) = U_i$.

This finiteness condition plays an important role in constructing dynamical partitions, as attested by the following lemma:

\begin{lemma} \label{lem:refinement} Let $\mathcal{V} = \{V_0, \ldots, V_n\}$ be a clopen partition of $X$ and suppose that there are only finitely many itineraries with respect to $\mathcal{V}$. Let $\mathcal{I}$ be the set of those itineraries. Consider the collection $\mathcal{U} := \{U(I) : I \in \mathcal{I}\}$, where $U(I)$ contains the points of $X$ that follow the itinerary $I$. Then $\mathcal{U}$ is a dynamical partition of $X$ that refines $\mathcal{V}$.

Conversely, if $\mathcal{V}$ has a refinement $\mathcal{U}$ that is a dynamical partition, then the set of itineraries with respect to $\mathcal{V}$ is finite.
\end{lemma}
\begin{proof} Clearly $\mathcal{U}$ is a finite (because $\mathcal{I}$ is finite by assumption) partition of $X$. If the itinerary $I$ reads $a_0 a_1 \ldots$ then $U(I)$ admits the description $U(I) = \bigcap_{k \geq 0} f^{-k}(V_{a_k})$, which exhibits it as an intersection of closed sets and shows that it is closed. Since the $U(I)$ are finite in number, partition $X$, and are all closed, it follows that they are all open. Thus $\mathcal{U}$ is a partition of $X$ into clopen sets. Notice that $p \in V_i$ if and only if its itinerary with respect to $\mathcal{V}$ begins with an $i$. Thus $V_i$ is the union of $U(I)$ where $I$ ranges over all the itineraries in $\mathcal{I}$ that begin with the symbol $i$ and so $\mathcal{U}$ refines $\mathcal{V}$. Finally, observe that if $p$ follows an itinerary $I$ then $f(p)$ follows the itinerary $\sigma(I)$ so $f(U(I)) \subset U(\sigma(I))$. Thus $\mathcal{U}$ is a dynamical partition.

The proof of the converse is similarly easy.
\end{proof}

Sometimes we will consider partitions of the form $\mathcal{U} = \{X \setminus U, U\}$, where $U$ is a clopen subset of $X$, and speak of the itinerary with respect to $U$ rather than $\mathcal{U}$. Labelling the elements in $\mathcal{U}$ as $U_0 := X \setminus U$ and $U_1 := U$, the itinerary of a point $p \in X$ with respect to $U$ is just a sequence of zeroes and ones that records $f^k(p) \in U$ with a $1$ and $f^k(p) \not\in U$ with a $0$. Equivalently, the itinerary of $p$ is just the sequence $\chi_U(f^k(p)) = \chi_{f^{-k}U}(p)$. This seemingly pedantic expression will be useful later on.

We conclude this subsection with two extension lemmas whose interest will become clear later on. Recall from the Introduction that a closed set $Y \subset X$ is said to be stable with respect to clopen sets if it has a basis of clopen positively invariant neighbourhoods.

\begin{lemma}[Extending a dynamical $\epsilon$--partition I]\label{lem:extend} 
Let $Z \subset X$ be stable with respect to clopen sets and suppose that $\mathcal W = \{W_i\}_{i \in \Lambda}$ is a dynamical $\epsilon$--partition of $Z$ (with the induced topology). Then, there exists a dynamical $\epsilon$--partition $\mathcal V = \{V_i\}_{i \in \Lambda}$ of a clopen, positively invariant neighborhood of $Z$ in $X$ such that $V_i \cap Z = W_i$.
\end{lemma}
\begin{proof} Denote by $\tau$ the map in $\Lambda$ that satisfies $f(W_i) \subset W_{\tau(i)}$. Since $X$ is totally disconnected, for each $i \in \Lambda$ we can choose a clopen neighbourhood (in $X$) of $W_i$, say $W'_i$, of diameter smaller than $\epsilon$ and such that $f(W'_i) \cap W'_j \neq \emptyset$ only for $j = \tau(i)$.

Since $\cup_i W'_i$ is a neighbourhood of $Z$ in $X$ and $Z$ is stable with respect to clopen sets, there exists a clopen, positively invariant neighbourhood $P$ of $Z$ contained in $\cup_i W'_i$. Define $V_i = W'_i \cap P$. As an intersection of clopen sets, $V_i$ is clopen in $X$, and $\{V_i\}$ constitutes a partition of $P$. Notice also that the choice of $W'_i$ guarantees that $\diam(V_i) < \epsilon$. It only remains to check that $\{V_i\}$ is a dynamical partition. Consider any $i \in \Lambda$. On the one hand, $V_i \subset P$ so $f(V_i) \subset f(P) \subset P = \cup_j V_j$, where we have used that $P$ is positively invariant. On the other, $V_i \subset W'_i$ so $f(V_i) \subset f(W'_i)$ and therefore $f(V_i) \cap V_j \subset f(W'_i) \cap W'_j$, which is empty unless $j = \tau(i)$. Since we have just seen that $f(V_i)$ is contained in the union of the $V_j$'s, it follows that $f(V_i) \subset V_{\tau(i)}$. 
\end{proof}

\begin{lemma}[Extending a dynamical $\epsilon$--partition II]\label{lem:trim}
Assume $P$ is a clopen positively invariant subset of $X$ and $\mathcal V = \{V_i\}$ is a dynamical $\epsilon$--partition of $P$. Then, we can add clopen sets of $X$ to $\mathcal V$ to obtain a dynamical $\epsilon$--partition of $f^{-1}(P)$.
\end{lemma}
\begin{proof}
Note that $f^{-1}(P) \setminus P$ is clopen in $X$ and is naturally partitioned in the sets $\{f^{-1}(V_i) \setminus P\}$, which are also clopen in $X$. Since $X$ is totally disconnected, each of these can be partitioned into clopen sets of diameter smaller than $\epsilon$. Adding them to $\mathcal V$ we obtain a dynamical $\epsilon$--partition of $f^{-1}(P)$.
\end{proof}

\section{Proof of Proposition \ref{prop:coh}} \label{sec:coh}

In this brief section we prove Proposition \ref{prop:coh}, which states that (if any) the nonzero eigenvalues of $f^{*}\colon \check{H}^{0}(X;\mathbb{C})\to\check{H}^{0}(X;\mathbb{C})$ are all roots of unity.

Recall that $\check{H}^{0}(X)$ can be identified with the $\mathbb{C}$--vector space of all locally constant functions on $X$.
Under this identification, the action of the induced homomorphism $f^{*}\colon \check{H}^{0}(X)\to\check{H}^{0}(X)$ is simply $\varphi\mapsto\varphi\circ f$. (See the Appendix for more details).

Now suppose that $\lambda\in\mathbb{C}$ is a nonzero eigenvalue of $f^{*}\colon\check{H}^{0}(X)\to\check{H}^{0}(X)$.
Then there exists a nonzero, locally constant map $\varphi:X\to\mathbb{C}$ such that $f^{*}(\varphi)=\lambda\varphi$; that is, $\varphi\circ f=\lambda\varphi$.
Using the fact that $\varphi$ is locally constant and $X$ is compact it is easy to see that $\varphi$ takes only finitely different values $c_{0},c_{1},\ldots,c_{n}$, where at least one of them is nonzero because $\varphi$ is nonzero, and the collection $\mathcal{U}:=\left\{U_{i}:=\varphi^{-1}(c_{i})\right\}$ constitutes a partition of $X$ into clopen sets.
From $\varphi\circ f=\lambda\varphi$, for any $i$ we have that $\varphi f(U_{i})=\lambda\varphi(U_{i})=\lambda c_{i}$, so $\varphi$ is constant over $f(U_{i})$ and there exists $j$ with $c_{j} = \lambda c_i$ and $f(U_{i})\subset U_{j}$. Therefore, in the terminology introduced earlier, $\mathcal{U}$ is a dynamical partition of $X$.
Recall from Section \ref{sec:partitions} that this automatically yields that there are only finitely many itineraries with respect to $\mathcal U$ and in particular there exist $k$ and $s$ such that the itinerary of any point is $s$--periodic.
Take $p$ such that $\varphi(p) \neq 0$. The previous statement and the definition of $\mathcal U$ imply that $\varphi(f^{k+s}(p)) = \varphi(f^k(p))$, so that $\lambda^{k+s} \varphi(p) = \lambda^k \varphi(p)$ and consequently $\lambda^s = 1$.

\section{Eigenvalues and $\epsilon$-partitions} \label{sec:eigen}

In this section we prove Theorem \ref{thm:intro}. It will be convenient to prove the theorem in this slightly different form:

\begin{theorem}\label{teo:epsilonparticiones}

Let $X$ be compact and totally disconnected and let $f \colon X \to X$ be continuous. Consider the induced map $f_* : \check{H}_0(X;\mathbb{C}) \to \check{H}_0(X;\mathbb{C})$. Then the following statements are equivalent:
\begin{itemize}
	\item[({\it i}\/)] There exists $\lambda \in \mathbb{C}$ with $|\lambda| \neq 0, 1$ that is not an eigenvalue of $f_*$.
	\item[({\it ii}\/)] No $\lambda \in \mathbb{C}$ with $|\lambda| \neq 0, 1$ is an eigenvalue of $f_*$.
	\item[({\it iii}\/)] $(X, f)$ admits dynamical $\epsilon$--partitions for every $\epsilon > 0$.
	\item[({\it iv}\/)] The number of different itineraries with respect to every clopen partition of $X$ is finite.
	\item[({\it v}\/)] The number of different itineraries with respect to every clopen subset of $X$ is finite.
\end{itemize}
\end{theorem}




The part of Theorem \ref{teo:epsilonparticiones} that requires more effort is the proof of $(i)$ $\Rightarrow$ $(iii)$. We will therefore address it first. Since the argument is slightly intricate we give a brief outline here. The main technical difficulty lies in the set $\Per_r(f)$, which is defined as the set of $f$--periodic points with period bounded above by certain natural number $r$. It will be easy to produce the desired partition \emph{away} from this set and also \emph{on} this set, but matching them will require some work. First we shall consider the case when $f$ is surjective. An arithmetical argument will produce a natural number $r(\lambda)$ that only depends on $\lambda$ and, for any $r \geq r(\lambda)$, we shall show that:

\begin{itemize}
	\item[(1)] Restricting our attention to the dynamical system $(\Per_r(f), f|_{\Per_r(f)})$, the latter has a dynamical $\epsilon$--partition.
	\item[(2)] This can be extended to a dynamical $\epsilon$--partition of a clopen, positively invariant neighbourhood $P$ of $\Per_r(f)$.
	\item[(3)] Letting $A$ be the set of points in $X$ whose forward orbit eventually enters $P$ (and remains there thereafter, since $P$ is positively invariant), the partition in (ii) can be extended to a dynamical $\epsilon$--partition of $A$.
	\item[(4)] There is a dynamical partition of $X \setminus P$ that is an ``$\epsilon$--partition modulo $P$'': every element of the partition either has diameter less than $\epsilon$ or, if not, is contained in $P$.
	\item[(5)] Taking the common refinement of the partitions in (3) and (4) yields a dynamical $\epsilon$--partition of all of $X$.
\end{itemize}

The proof when $f$ is not surjective will build on the surjective case. We will consider the smallest invariant set in which $f$ is surjective, which is $Y =  \cap_{n \ge 0} f^n(X)$, and prove that:
\begin{itemize}
	\item[(6)] The hypotheses of the theorem still hold for the restriction $f|_Y : Y \to Y$ and so (since $f|_Y$ is surjective) there is a dynamical $\epsilon$--partition of $Y$.
	\item[(7)] The partition in (6) can be extended to a clopen, positively invariant neighbourhood of $Y$ in $X$.
	\item[(8)] The partition in (7) can be extended to a dynamical $\epsilon$--partition of the whole $X$.
\end{itemize}

As the reader can see, extending dynamical partitions is a key step in the proofs. This is where Lemmas \ref{lem:extend} and \ref{lem:trim} from Section \ref{sec:partitions} will become useful. The first extends dynamical $\epsilon$--partitions from a closed set (which needs to satisfy some additional hypothesis) to a clopen, positively invariant neighbourhood; this is used in going from (1) to (2) and also from (6) to (7). The second extends dynamical $\epsilon$--partitions of a clopen positively invariant set to its preimage under $f$; we use it to go from (2) to (3) and also from (7) to (8).

\subsection{Proof of ({\it i}\/) $\Rightarrow$ ({\it iii}\/) for surjective $f$}

First, we need a technical lemma. This does not involve dynamics or topology. For any $r = 0,1,2,\ldots$ denote by $S_r$ the set of sequences $(a_k)$ of zeroes and ones such that whenever a term of the sequence is one, the following $r$ terms (at least) are zero. That is, if $a_k = 1$ then $a_{k+1} = \ldots = a_{k+r} = 0$. $S_r$ is a subshift of finite type of the full one--sided shift $\{0, 1\}^{\mathbb N}$.

\begin{lemma} \label{lem:unique} Let $\lambda \in \mathbb{C}$ have modulus $|\lambda| > 1$. For big enough $r$, whenever two sequences $(a_k)$ and $(b_k)$ belonging to $S_r$ satisfy \[\sum_{k \geq 0} a_k \lambda^{-k} = \sum_{k \geq 0} b_k \lambda^{-k}\] then the sequences themselves must be equal.
\end{lemma}
\begin{proof} Suppose that the sequences $(a_k)$ and $(b_k)$ differ. By cancelling the (potentially empty) initial block of terms where the sequences coincide we may assume without loss of generality that $a_0 \neq b_0$ so that the difference $a_0 - b_0 = \pm 1$. Taking modulus and rearranging terms we have \begin{equation} \label{eq:sum} 1 \le \sum_{k \geq 1} \frac{|a_k-b_k|}{|\lambda|^k} = \sum_{k \geq 1} \frac{c_k}{|\lambda|^k}.\end{equation} where we have defined $c_k := |a_k - b_k|$. Notice that $(c_k)$ is a sequence of zeroes and ones with $c_0 = 1$. The condition that $(a_k)$ and $(b_k)$ belong to $S_r$ implies that the following property holds: any block $B$ of $r$ consecutive terms of $(c_k)$ contains, at most, two nonzero entries (which are, therefore, ones).

Now think of the sequence $(c_k)$ grouped in blocks of $r$ terms thus: \[c_0 c_1 \dots c_{r-1} \quad , \quad c_r c_{r+1} \ldots c_{2r-1} \quad , \quad c_{2r} c_{2r+1} \ldots c_{3r-1} \quad , \quad \ldots\]

As just mentioned each of these blocks contains at most two ones; the remaining terms being zero. Moreover, the first block begins with $c_0 = 1$ so among $c_1 \ldots c_{r-1}$ there is at most one nonzero term. The contribution of $c_1 \ldots c_{r-1}$ to the series in Equation \eqref{eq:sum} is therefore bounded above by $\nicefrac{1}{|\lambda|}$. The contribution of the remaining blocks, each of which contains at most two ones, attains its maximum value when the two ones appear in the first two positions of the block; thus, their contributions are bounded above by $\nicefrac{1}{|\lambda|^r} + \nicefrac{1}{|\lambda|^{r+1}}$, by $\nicefrac{1}{|\lambda|^{2r}} + \nicefrac{1}{|\lambda|^{2r+1}}$, and so on. Putting all this together, the series in Equation \eqref{eq:sum} can be bounded above by \[\left( \frac{1}{|\lambda|} \right) + \left( \frac{1}{|\lambda|^r} + \frac{1}{|\lambda|^{r+1}} \right) + \left( \frac{1}{|\lambda|^{2r}} + \frac{1}{|\lambda|^{2r+1}} \right) + \ldots = \left(\frac{1}{|\lambda|} + \frac{1}{|\lambda|^r} \right) \frac{1}{1-\frac{1}{|\lambda|^r}}.\] To arrive at a contradiction we need to choose $r$ so that the right hand side of the above is $< 1$, because then Equation \eqref{eq:sum} is violated. Imposing this condition and rearranging terms $r$ must be chosen to satisfy \[\frac{1}{|\lambda|} + \frac{1}{|\lambda|^r} + \frac{1}{|\lambda|^r} < 1,\] which will certainly hold for big enough $r$ since $|\lambda| > 1$.
\end{proof}

{\it Notation.} From now on we will write $r(\lambda)$ to denote any number big enough so that Lemma \ref{lem:unique} is satisfied.
\medskip

\emph{To simplify subsequent writing, for the results in this subsection the notation and assumptions are as in Theorem \ref{teo:epsilonparticiones}, together with the hypothesis that $f$ is surjective.}

\begin{proposition} \label{prop:finite} Let $r \geq r(\lambda)$ and let $V$ be a clopen set having the following property: $V$ is disjoint from $fV, f^2V, \ldots, f^rV$. Then there are only finitely many different itineraries with respect to $\mathcal{V} := \{X \setminus V, V\}$.
\end{proposition}
\begin{proof} By assumption there exists $\lambda \in \mathbb{C}$ with $|\lambda| \neq 0,1$ that is not an eigenvalue of $f_*$. As explained in the Appendix, this implies that $f^* - \lambda {\rm Id}$ is surjective, where $f^*$ is the homomorphism induced by $f$ in $\check{H}^0(X)$. In particular, there exists a locally constant function $\psi \in \check{H}^0(X)$ such that $(f^* - \lambda \id)(\psi) = \chi_V$, where $\chi_V$ is the characteristic function of $V$. Unravelling the notation, this means that $\psi f - \lambda \psi = \chi_V$.

By the surjectivity of $f$, for any point $p \in X$ there exists a full orbit $\{p_n\}_{n = -\infty}^{+\infty}$ through $p$ (that means $p_0 = p$ and $f(p_n) = p_{n+1}$ for every $n$). The eigenvalue equation yields two possible expansions for $\psi(p)$: one in terms of the backward semiorbit (first equation) and another in terms of the forward semiorbit (second equation):
\begin{align*}
\psi(p)&= \chi_V(p_{-1}) + \lambda \psi(p_{-1}) = \ldots = \sum_{k = 0}^{n-1} \lambda^k \chi_V(p_{-(k+1)}) + \lambda^n \psi(p_{-n}) \\
\psi(p)&= - \frac{1}{\lambda} \chi_V(p_0) + \frac{1}{\lambda} \psi(p_1) = \ldots = -\sum_{k = 0}^{n-1} \frac{1}{\lambda^{k+1}} \chi_V(p_{k}) + \frac{1}{\lambda^n} \psi(p_n)
\end{align*}

Now, because $\psi$ is locally constant and $X$ is compact, $\psi$ takes only finitely many values and in particular it is bounded. Thus we may take $n \to +\infty$ in the equalities above and conclude that $\psi(p)$ can be expressed as
\begin{equation}\label{eq:1}
\frac{1}{\lambda}\sum_{k = 1}^{+\infty} \lambda^k \chi_V(p_{-k})
\enskip \text{ or } \enskip
 -\frac{1}{\lambda}\sum_{k = 0}^{+\infty} \lambda^{-k} \chi_V(p_{k})
\end{equation}
depending on whether $|\lambda| < 1$ or $|\lambda| > 1$, respectively. It follows that the power series (actually only one of them is well-defined depending on $\lambda$) only takes a finite number of values as we evaluate all possible semiorbits of $f$. 

The condition on $V$ in the statement ensures that the sequence of coefficients of both power series in consideration in (\ref{eq:1}) belongs to $S_r$. Thus, Lemma \ref{lem:unique} can be applied to conclude that:
\begin{itemize}
\item[(a)] If $|\lambda| > 1$ there are only finitely many different sequences of the form $(\chi_V(p_0), \chi_V(p_1), \ldots)$ as $\{p_n\}_{n \ge 0}$ ranges over all forward semiorbits in $X$.
\item[(b)] If $|\lambda| < 1$ there are only finitely many different sequences of the form $(\chi_V(p_{-1}), \chi_V(p_{-2}), \ldots)$ as $\{p_n\}_{n \le -1}$ ranges over all backward semiorbits in $X$.
\end{itemize}

In the first case ($|\lambda| > 1$) we directly obtain the existence of a finite number of different itineraries with respect to $\mathcal V$. To address the case $|\lambda| < 1$ note that, for a given backward semiorbit $\{p_n\}_{n \le -1}$, the sequence $\{p_n\}_{n \le -k}$ is also a backward semiorbit (for the point $p_{-k+1}$) for every $k \geq 1$. By (b), there must exist two of these backward semiorbits for which the sequences $(\chi_V(p_n))$ coincide. Thus if we denote by $q$ the (finite) number of values attained by $\psi$, there exist positive integers $1 \leq s < s' \leq q+1$ such that 
\[
(\chi_V(p_{-s}), \chi_V(p_{-(s+1)}), \ldots) = (\chi_V(p_{-s'}), \chi_V(p_{-(s'+1)}), \ldots).
\]
Consequently, both sequences and also $(\chi_V(p_{-1}), \chi_V(p_{-2}), \ldots)$ are periodic of period $s' - s$ and, in particular, periodic of period at most $q$. The same statement then carries over to forward semiorbits as well. Then, all itineraries of points $p$ with respect to $\mathcal V$ are periodic of period at most $q$ so there are only finitely many of them.
\end{proof}

In view of the previous lemma, the set $\Per_r(f)$ that consists of all the periodic points of $X$ whose period is bounded by $r$ plays an importante role: any $p \notin \Per_r(f)$ has a neighborhood to which we can apply Proposition \ref{prop:finite}.

\begin{proposition} \label{prop:e-part_1} Let $r \geq r(\lambda)$ and let $O$ be a clopen neighborhood of $\Per_r(f)$. Then, for any $\epsilon >0$ there exists a dynamical partition $\mathcal{U}$ of $X$ whose elements of diameter greater than $\epsilon$ are contained in $O$. Moreover:
\begin{itemize}
	\item[(i)] $O$ is a union of elements of $\mathcal{U}$.
	\item[(ii)] There exists an element $U_* \in \mathcal{U}$ which is positively invariant and satisfies $\Per_r(f) \subset U_* \subset O$.
\end{itemize}
\end{proposition}

Thus, although $\mathcal{U}$ is not necessarily a dynamical $\epsilon$--partition because it may have elements of diameter bigger than $\epsilon$, the latter are controlled in the sense that they are contained in $O$.

\begin{proof} Since $X \setminus O$ is disjoint from ${\rm Per}_{r}(f)$, each point $p \notin O$ satisfies that $p$ is different from $f(p), \ldots, f^r(p)$ (notice however that, if $f$ is not injective, the latter need not be all different). In particular $p$ has a neighbourhood $V_p$ such that $V_p$ is disjoint from $fV_p, \ldots, f^rV_p$. We may also assume $V_p$ to be disjoint from $O$ (since $O$ is closed), have diameter smaller than $\epsilon$ and also to be a clopen set, since $X$ is totally disconnected. Doing this for every $p \notin O$ yields a cover of $X \setminus O$, of which we may extract a finite subcover which we relabel as $V_1, \ldots, V_m$. Consider, for each of them, the clopen partition $\mathcal{V}_i := \{X \setminus V_i,V_i\}$ of $X$. By Proposition \ref{prop:finite} the set of itineraries with respect to each $\mathcal{V}_i$ is finite. Let $\mathcal{V} := \mathcal{V}_1 \vee \ldots \vee \mathcal{V}_m$. This is a clopen partition of $X$ and there are only finitely many itineraries with respect to it by the remark before Lemma \ref{lem:refinement}. Notice that $O$ belongs to $\mathcal{V}$, since it is the intersection $(X \setminus V_1) \cap \ldots \cap (X \setminus V_m) = X \setminus (V_1 \cup \ldots \cup V_m)$. Any other element $V$ in $\mathcal{V}$ is an intersection of elements from the various $\mathcal{V}_i$ where at least one of them is $V_i$ (rather than $X \setminus V_i$), which implies that $V \subset V_i$ and so it has diameter less than $\epsilon$. Finally, let $\mathcal{U}$ be the partition constructed from $\mathcal{V}$ as in Lemma \ref{lem:refinement}, by letting $U(I)$ be the set of points that follow itinerary $I$ with respect to $\mathcal{V}$. This is a clopen dynamical partition that refines $\mathcal{V}$. In particular, if $U \in \mathcal{U}$ has diameter bigger than $\epsilon$ then the element $V \in \mathcal{V}$ that contains $U$ also has diameter bigger than $\epsilon$ and so it must be $V = O$; that is, $U \subset V = O$.

Let us check that (i) and (ii) hold. The fact that $O$ is a union of elements of $\mathcal{U}$ is a direct consequence of Lemma \ref{lem:refinement}.(b). As for (ii), observe the following. By assumption ${\rm Per}_r(f) \subset O$ and, since ${\rm Per}_r(f)$ is positively invariant under $f$, the forward orbit of each point in ${\rm Per}_r(f)$ remains in $O$. Thus the itinerary of all the points of ${\rm Per}_r(f)$ with respect to $\mathcal{V}$ is the same, namely the constant sequence $I_0 = * * \ldots$ where $*$ is the label of $O$ in the partition $\mathcal{V}$. Consequently ${\rm Per}_r(f)$ is contained in the single element $U(I_0)$ of $\mathcal{U}$, which is in turn contained in $O$. Moreover, $f(U(I_0)) \subset U(\sigma(I_0)) = U(I_0)$, so $U(I_0)$ is positively invariant. Setting $U_* := U(I_0)$ finishes the proof.
\end{proof}

Since $O$ was arbitrary in the previous proposition and $U_*$ is a clopen neighbourhood of $\Per_r(f)$, as a consequence we obtain the following corollary:

\begin{corollary}\label{cor:invariantnbd} $\Per_r(f)$ is stable with respect to clopen sets.
\end{corollary}

(We recall that this terminology was introduced just before Lemma \ref{lem:extend} and means that $\Per_r(f)$ has a basis of clopen, positively invariant neighborhoods).

\begin{proposition} \label{prop:finite_union} Let $P$ be a clopen, positively invariant neighbourhood of $\Per_r(f)$. Let $A$ be the set of points $p \in X$ whose positive semiorbit eventually enters $P$ (and remains there thereafter, since $P$ is positively invariant). Then there exists $n_0$ such that $f^n(p) \in P$ for every $n \geq n_0$ and every $p \in A$.
\end{proposition}

More briefly stated: all the points in $X$ that eventually enter $P$ will have done so after a number of iterates that is independent of $p$.

\begin{proof} Apply Proposition \ref{prop:e-part_1} to $O := P$ to obtain the corresponding dynamical partition $\mathcal{U}$. Then $P$ is a union of elements of $\mathcal{U}$, and we may assume without loss of generality that these correspond to the labels $0, 1, \ldots, s$. Thus, points in $A$ are characterized as those whose itinerary with respect to $\mathcal{U}$ contains at least one appearance of the symbols $0, 1, \ldots, s$. For each $p \in A$ let $k(p)$ be the first position in its itinerary where a symbol among $0, 1, \ldots, s$ appears or, equivalently, let $k(p)$ be the first iterate of $p$ that belongs to $P$. Since there are only finitely many itineraries with respect to $\mathcal{U}$ because $\mathcal{U}$ is a dynamical partition, the set of all $k(p)$ as $p$ runs in $A$ is finite, and we may set $n_0$ equal to the maximum of them. Then $f^n(p) \in P$ whenever $n \geq n_0$ for every $p \in A$.
\end{proof}

\begin{proposition} \label{prop:finitePer} For every $\epsilon> 0$ there exists a clopen, positively invariant neighbourhood $P$ of $\Per_r(f)$ that has a dynamical $\epsilon$--partition.
\end{proposition}
\begin{proof} The set $\Per_r(f)$ is invariant, so we may consider the restriction $(\Per_r(f), f|_{\Per_r(f)})$. Let $\mathcal{W}_0$ be an arbitrary partition of $\Per_r(f)$ by clopen sets (with the induced topology) of diameter less than $\epsilon$. This exists because $\Per_r(f)$ is closed in $X$ and therefore compact and totally disconnected. Since the period of any point in $\Per_r(f)$ is bounded above by $r$, the itineraries of points with respect to $\mathcal{W}_0$ are periodic of period at most $r$. This readily implies that there are only finitely many itineraries with respect to $\mathcal{W}_0$ and so by Lemma \ref{lem:refinement} (applied to the restriction $(\Per_r(f), f|_{\Per_r(f)})$) there is a dynamical partition $\mathcal{W}$ of $\Per_r(f)$ that refines $\mathcal{W}_0$. In particular, it is an $\epsilon$--partition. Corollary \ref{cor:invariantnbd} entitles us to apply Lemma \ref{lem:extend} to $\Per_r(f)$ and $\mathcal{W}$ to obtain a dynamical $\epsilon$--partition of a clopen, positively invariant neighbourhood $P$ of $\Per_r(f)$.
\end{proof}

Now we are ready to prove Theorem \ref{teo:epsilonparticiones} in the surjective case. Choose any $r \geq r(\lambda)$. By Proposition \ref{prop:finitePer} the set $\Per_r(f)$ has a clopen, positively invariant neighbourhood $P$ that has a dynamical $\epsilon$--partition $\mathcal{V}$. Consider the set $A$ of points whose positive semiorbit eventually enter $P$. According to Proposition \ref{prop:finite_union} there exists $n_0$ such that $A = f^{-n_0}(P)$. Applying Lemma \ref{lem:trim} consecutively to $P$, then to $f^{-1}(P)$ and so on up to $f^{-n_0+1}(P)$ yields a dynamical $\epsilon$--partition of $A$. Notice that $X \setminus A$ is a positively invariant clopen set, so adjoining it to the partition just obtained gives a dynamical partition $\mathcal{U}$ of the whole $X$ such that every element of $\mathcal{U}$ other than $X \setminus A$ has diameter less than $\epsilon$.

We are almost finished. By Proposition \ref{prop:e-part_1} there is a dynamical partition $\mathcal{U}'$ of $X$ such that every element in $\mathcal{U}'$ having diameter bigger than $\epsilon$ is contained in $P$. Then the common refinement $\mathcal{U} \vee \mathcal{U}'$ is a dynamical partition of $X$, and it is in fact an $\epsilon$--partition. The reason is that for every $p \in X$ either the element of $\mathcal{U}$ or the element of $\mathcal{U}'$ that contain it have diameter less than $\epsilon$ by construction.

\subsection{Proof of ({\it i}\/) $\Rightarrow$ ({\it iii}\/) for an arbitrary $f$}

Let us address the general case. Since $f$ need not be surjective, we cannot apply Proposition \ref{prop:finite} and subsequent results directly to $f$. The idea is now to restrict our considerations to the smallest invariant set in which $f$ is surjective, which is $Y =  \cap_{n \ge 0} f^n(X)$. Building on the previous case we will be able to construct dynamical $\epsilon$--partitions of $Y$ which we shall later on extend to dynamical $\epsilon$--partitions of the whole space $X$.

More specifically, we will prove the following two lemmas:

\begin{lemma} \label{lem:eigen}
Denote by $f|_Y \colon Y \to Y$ the map $f$ with both domain and image restricted to $Y$. If $\lambda \neq 0$ is an eigenvalue of $(f|_Y)_* \colon \check{H}_0(Y) \to \check{H}_0(Y)$, then it is also an eigenvalue of $f_* \colon \check{H}_0(X) \to \check{H}_0(X)$.
\end{lemma}

\begin{lemma}\label{lem:goodbasis} Let $X$ be a compact totally disconnected space and $f \colon X \to X$ continuous. Then $Y = \cap_{n \ge 0} f^n(X)$ is stable with respect to clopen sets.
\end{lemma}

Let us see how the general case of Theorem \ref{teo:epsilonparticiones} follows from these. As a consequence of Lemma \ref{lem:eigen}, the assumption of Theorem \ref{teo:epsilonparticiones} implies that there exists $\lambda \in \mathbb C$ with $|\lambda| \neq 0, 1$ that is not an eigenvalue of $(f|_Y)_* \colon \check{H}_0(Y) \to \check{H}_0(Y)$. Since $Y$ is compact and totally disconnected and $f|_Y : Y \to Y$ is surjective, we may apply the results of the previous section to conclude that for every $\epsilon > 0$ there exists a dynamical $\epsilon$--partition $\mathcal{W}$ of $Y$. Now, since $Y$ is stable with respect to clopen sets by Lemma \ref{lem:goodbasis}, we may apply Lemma \ref{lem:extend} to $\mathcal W$ and obtain a dynamical $\epsilon$--partition of a clopen neighborhood $P$ of $Y$ in $X$. The compactness of $X$ ensures that $P$ contains $f^m(X)$ for some integer $m$, so $f^{-m}(P) = X$. Thus, if we apply $m$ times Lemma \ref{lem:trim} consecutively to $P, f^{-1}(P), \ldots, f^{-m+1}(P)$ and $\mathcal{W}$ and the subsequent augmented partitions we obtain a dynamical $\epsilon$--partition of the whole $X$ as desired, finishing the proof of Theorem \ref{teo:epsilonparticiones} in full generality.

It remains to prove the two auxiliary lemmas stated above.

\begin{proof}[Proof of Lemma \ref{lem:eigen}]  Let $i \colon Y \subset X$ be the inclusion. Consider an eigenvector $u \in \check{H}_0(Y)$ associated to $\lambda$ and set $v := i_*(u) \in \check{H}_0(X)$. Since $if|_Y = f i$, it is straightforward to check that $v$ satisfies the condition $f_*(v) = \lambda v$, so to prove the lemma we only need to show that $v$ is nonzero.

The inclusions $X \supset f(X) \supset f^2(X) \supset \ldots$ induce homomorphisms $\check{H}^0(X) \to \check{H}^0(f(X)) \to \check{H}^0(f^2(X)) \to \ldots$ Thinking of the elements of $\check{H}^0$ as locally constant maps, these inclusion induced homomorphisms are simply given by restrictions: $\varphi \mapsto \varphi|_{f(X)} \mapsto \varphi|_{f^2(X)} \mapsto \ldots$

By the continuity property of \v{C}ech cohomology $\check{H}^0(Y) = \varinjlim \check{H}^0(f^nX)$. Then observe that \[\check{H}_0(Y) = \Hom(\check{H}^0(Y);\mathbb{C}) = \Hom(\varinjlim \check{H}^0(f^nX);\mathbb{C}) = \varprojlim \Hom(\check{H}^0(f^nX);\mathbb{C}) = \varprojlim \check{H}_0(f^nX)\] where in all limits the bonding maps are induced by inclusions (here we have tacitly used several results from algebraic topology that are succintly explained in the Appendix). Since $u \neq 0$ and $\check{H}_0(Y) = \varprojlim \check{H}_0(f^nX)$, there must exist some $n$ such that the inclusion $j \colon Y \subset f^nX$ carries $u$ onto a nonzero vector.

Consider picking a point $p \in Y$, acting on it with $f^n$ and then viewing the resulting point as an element in $f^nX$. This can be written as the composition of two maps in two different ways: \[Y \stackrel{i}{\longrightarrow} X \stackrel{f^n}{\longrightarrow} f^nX\] and \[Y \stackrel{(f|_Y)^n}{\longrightarrow} Y \stackrel{j}{\longrightarrow} f^nX\] (attention should be paid to the source and target spaces of each map). In particular $(f^n)_* i_* = j_* (f|_Y)^n_*$ and acting with this on $u$ we have \[(f^n)_* i_*(u) = j_* (f|_Y)^n_*(u) = j_*(\lambda^n u) = \lambda^n j_*(u)\] The right hand side of this expression is nonzero, since we argued earlier that $j_*(u) \neq 0$ and $\lambda \neq 0$ by assumption. Thus $i_*(u)$ must also be nonzero, as we wanted to prove.
\end{proof}

\begin{proof}[Proof of Lemma \ref{lem:goodbasis}] Let $O$ be a clopen neighbourhood of $Y$ and take $n_0$ big enough so that $f^{n_0}(X) \subset O$ (this exists by the compactness of $X$). Now consider the set $O'$ of points whose first $n_0$ iterates (including the zeroth one) belong to $O$; that is, \[O' := O \cap f^{-1}(O) \cap \ldots \cap f^{-(n_0-1)}(O).\] This is a clopen set since it is a finite intersection of clopen sets. Clearly $Y \subset O' \subset O$, so we only need to check that $O'$ is positively invariant. Pick $p \in O'$. Since its first $n_0$ iterates belong to $O$, the first $n_0-1$ iterates of $f(p)$ also belong to $O$. The $n_0$th iterate of $f(p)$ satisfies $f^{n_0-1}(f(p)) = f^{n_0}(p) \in f^{n_0}(X) \subset O$ by the choice of $n_0$, so it also belongs to $O$ and therefore $f(p)$ belongs to $O'$ indeed.
\end{proof}

\subsection{Proof of (iii) $\Rightarrow$ (iv) $\Rightarrow$ (v) $\Rightarrow$ (ii)} Let us show first that ({\it iii}\/) $\Rightarrow$ ({\it iv}\/). Take an arbitrary clopen and finite partition $\mathcal{U}$ of $X$ and denote $\delta$ its Lebesgue number. By hypothesis, there exists a dynamical $\delta'$--partition $\mathcal{V}$ of $X$ for some $\delta' < \delta$ and this implies, in particular, that $\mathcal{V}$ is a dynamical refinement of $\mathcal{U}$. By the converse statement in Lemma \ref{lem:refinement}, the set of itineraries with respect to $\mathcal{U}$ is finite and we obtain (\textit{iv}). Implication ({\it iv}\/) $\Rightarrow$ ({\it v}\/) is obvious.

Finally we prove that ({\it v}\/) $\Rightarrow$ ({\it ii}\/). Suppose that $\lambda$ is an eigenvalue of $f_*$ and let $0 \neq T \in \check{H}_0(X)$ be an eigenvector for $\lambda$, so that $f_*(T) = \lambda T$. Since $\check{H}_0(X)$ is dual to $\check{H}^0(X)$, we may think of $T$ as a linear map $T \colon \check{H}^0(X) \to \mathbb{C}$. Recall that $\{\chi_U: U \text{ is clopen in }X\}$ is a generating system of $\check{H}^0(X)$ (see the Appendix) so, since $T \neq 0$, there must exist a clopen $U$ such that $T(\chi_U) \neq 0$. By assumption there are only finitely many itineraries with respect to $U$, so from the results in Section \ref{sec:partitions} there exist positive integers $k, s$ such that for every $p \in X$ the itinerary of $f^k(p)$ is $s$--periodic. Consider the sets $f^{-k}(U)$ and $f^{-(k+s)}(U)$. The first can be alternatively described as the set of points $p \in X$ such that the itinerary of $f^k(p)$ with respect to $U$ begins with a one; the second, as the set of points such that the itinerary of $f^{k+s}(p)$ begins with a one. But, because of the $s$--periodicity, these conditions are equivalent and so $f^{-k}(U) = f^{-(k+s)}(U)$. Then from
\begin{align*}
& \lambda^k T(\chi_U) = (f_*)^k T (\chi_U) = T((f^*)^k \chi_U) = T(\chi_{f^{-k}(U)}) \\ 
& \lambda^{k+s} T(\chi_U) = (f_*)^{k+s} T (\chi_U) = T((f^*)^{k+s} \chi_U) =T\left(\chi_{f^{-(k+s)}(U)}\right)
\end{align*}
\noindent
it follows that $\lambda^k = \lambda^{k+s}$ so $|\lambda| = 0,1$.

\section{$\epsilon$--partitions and dynamics. Proof of Theorem \ref{thm:intro2}}\label{sec:characterization}

The conclusion of Theorem \ref{teo:epsilonparticiones} naturally leads to the following question: what are the discrete dynamical systems $(X, f)$, where $X$ is compact and totally disconnected and $f \colon X \to X$ is continuous, that admit dynamical $\epsilon$--partitions for every $\epsilon > 0$? As a first step towards an answer let us prove the following:

\begin{proposition} \label{prop:finite_liminv} Let $X$ be compact and totally disconnected, and let $f : X \to X$ be continuous. Then $(X,f)$ admits dynamical $\epsilon$--partitions for every $\epsilon> 0$ if, and only if, $(X,f)$ is conjugate to the inverse limit of an inverse sequence of discrete dynamical systems $(F_n,\tau_n)$ where each $F_n$ is a finite space endowed with the discrete topology.
\end{proposition}

\begin{proof} ($\Rightarrow$) Suppose first that $(X,f)$ admits dynamical $\epsilon$--partitions for every $\epsilon > 0$. Then it is easy to construct inductively a sequence of dynamical partitions of $X$ by clopen sets whose diameters tend to 0 and such that each partition refine the previous one:
\[
\mathcal U_1 < \mathcal U_2 < \ldots < \mathcal U_n < \ldots
\]

Since each covering $\mathcal{U}_{n}$ refines the previous one, there are natural bonding maps $j_n \colon \mathcal{U}_{n+1} \to \mathcal{U}_n$. Also, since the $\mathcal{U}_n$ are dynamical partitions for every $n$ there is an action map $i_n \colon X \to \mathcal{U}_n$; furthermore, clearly these commute with the $j_n$ ($j_{n} \circ i_{n+1}=i_{n}$). Thus there is a commutative diagram \[\xymatrix{\mathcal{U}_1 & \mathcal{U}_2 \ar[l]_{j_1} & \mathcal{U}_3 \ar[l]_{j_2} & \ldots \ar[l] \\ X\ar[u]_{i_1} & X \ar[l] \ar[u]_{i_2} & X \ar[l] \ar[u]_{i_3} & \ldots \ar[l]}\] where the unlabeled arrows in the lower row denote the identity. Taking the inverse limits of the upper and lower rows and identifying the latter with $X$ itself we see that the $i_n$ induce a continuous map $i \colon X \to \varprojlim \mathcal{U}_n$. The elements of $\varprojlim \mathcal{U}_n$ are nested sequences $U_1 \supset U_2 \supset U_3 \supset \ldots$ where each $U_n$ is a member of $\mathcal{U}_n$. The map $i$ sends $p \in X$ to the nested sequence of $U_n \in \mathcal{U}_n$ such that $p \in U_n$ for every $n$. Conversely, any sequence $(U_n) \in \varprojlim \mathcal{U}_n$ uniquely determines a point $p \in \bigcap_{n \geq 1} U_n$, yielding a continuous map $\varprojlim \mathcal{U}_n \to X$ which is clearly the inverse of $i$. Thus $i$ is a homeomorphism between $X$ and $\varprojlim \mathcal{U}_n$.

Recall that each $\mathcal{U}_n$ carries a map $\tau_n \colon \mathcal{U}_n \to \mathcal{U}_n$ given by the dynamics in the partition; so that $f(U_n) \subseteq \tau_n(U_n)$ for any $U_n \in \mathcal{U}_n$. Since the bonding maps between the $\mathcal{U}_n$ commute with the $\tau_n$, the latter induce a continuous map $\tau$ on $\varprojlim \mathcal{U}_n$. Furthermore, the $i_n$ semiconjugate $f$ and $\tau_n$, so that their limit $i$ semiconjugates $f$ and $\tau$. We showed earlier that $i$ is a homeomorphism, so we conclude that it is actually a conjugation between $f$ and $\tau$. This proves ($\Rightarrow$), where the finite sets $F_n$ are just the partitions $\mathcal{U}_n$.

($\Leftarrow$) Since the property of having dynamical $\epsilon$--partitions is independent of the choice of metric and preserved by topological conjugacy, it will suffice to show that inverse limits of finite dynamical systems have this property. Thus, suppose we have an inverse sequence of finite spaces \[\xymatrix{F_1 & F_2 \ar[l]_{j_1} & F_3 \ar[l]_{j_2} & \ldots \ar[l]_{j_3}}\] where each $F_n$ is endowed with a map $\tau_n \colon F_n \to F_n$ which commutes with the bonding maps $j_n$. Let $\tau$ be the inverse limit of the maps $\tau_n$. We claim that $(\varprojlim F_n,\tau)$ has dynamical $\epsilon$--partitions for any $\epsilon> 0$. Denote by $\pi_n \colon \varprojlim F_n \to F_n$ the canonical projections. Consider, for every $n$, the family $\mathcal{U}_n := \{\pi_n^{-1}(e) : e \in F_n\}$. Clearly each $\mathcal{U}_n$ is a finite partition of $\varprojlim F_n$ by clopen sets (recall that $\pi_n$ are continuous and $F_n$ is endowed with the discrete topology). Moreover, each $\mathcal{U}_n$ is a dynamical partition because the behaviour of $\tau$ on the elements of $\mathcal{U}_n$ is conjugate (via the projection $\pi_n$) to the $\tau_n$ on $F_n$. The result follows from the fact that the family $\{\mathcal{U}_n: n \ge 1\}$ is cofinal among all clopen coverings of $X$ because the topology in $\varprojlim F_n$ is the initial topology with respect to the projections $\pi_n$.
\end{proof}

For the next few paragraphs suppose that $(X,f)$ indeed admits dynamical $\epsilon$--partitions for every $\epsilon> 0$, so that we can identify $X = \varprojlim F_n$ and $f = \tau$ in the notation of Proposition \ref{prop:finite_liminv}. This allows for a complete description of the closed invariant subsets of $(X, f)$ as follows. If $L$ is a closed $\tau$--invariant subset of $\varprojlim F_n$ (that is, $\tau(L) = L$), the same is true of its projection $L_n = \pi_n(L)$ onto $F_n$. Thus, $L_n$ is a finite union of periodic orbits and $L$ is the inverse limit of the inverse sequence determined by the sets $L_n$. Furthermore, we also obtain that $L$ is stable with respect to clopen sets because the sequence $\{\pi_n^{-1}(L_n): n \ge 1\}$ is a basis of clopen and invariant neighborhoods of $L$.

Recall that if $f$ is surjective then each $\tau_n$ is surjective as well; hence a bijection of $F_n$. Thus each $F_n$ consists of a disjoint union of $\tau_n$--periodic orbits. Given a point $p \in \varprojlim F_n$, for each $n$ we may consider the periodic orbit $P_n \subseteq F_n$ of the element $\pi_n(p)$ in $F_n$. The bonding maps between the $F_n$ restrict to bonding maps between the $P_n$, and the latter form an inverse sequence whose inverse limit is the smallest closed invariant subset that contains $p$; that is, the closure of the orbit of $p$. Such an inverse limit is known as an adding machine (if the periods of the orbits go to infinity) or a periodic orbit (if the periods of the orbits stabilize). We recall that adding machines are minimal sets, that is, every (full) orbit is dense in them. In fact, something even stronger is true: the positive and the negative semiorbit of any point in an adding machine is dense.

When $f$ is not surjective a similar picture emerges because every orbit of a map defined on a finite set is eventually periodic. Thus, the projection onto $F_n$ of the orbit of $p \in \varprojlim F_n$ under $\tau$ is an orbit that is eventually equal to a $\tau_n$--periodic orbit $P_n$. The $\omega$--limit of $p$ in $\varprojlim F_n$ is equal to the inverse limit of these $P_n$ as above. Thus we have shown the following:

\begin{remark}\label{rmk:omegalimit} The $\omega$--limit of every $p \in X$ is either a periodic orbit or an adding machine and it is stable with respect to clopen sets.
\end{remark}

Whereas above we have used the definition of adding machines as inverse limits of periodic orbits, it might be convenient to recall an alternative description that is somewhat more pictorial. Consider a sequence $b_1, b_2, \ldots$ of arbitrary positive integers that one should think of as bases (in the sense of elementary arithmetic) and let $\Lambda$ be the space of sequences of integers $(a_1,a_2,\ldots)$ subject to the condition that $0 \leq a_i < b_i$ (endowed with the product of the discrete topology on each entry). Each of these sequences can, heuristically, be thought of as some sort of multi-base expansion of an integer where the $i$th digit is referred to base $b_i$. Then the operation ``addition of one unit'' is represented by a map $\nu \colon \Lambda \to \Lambda$ that is defined in the natural way taking into account the ``carryover'' onto the next position whenever $b_i$ is reached at some position $i$. More explicitly, $\nu$ maps $(a_i)$ to $(a'_i)$ where
\begin{itemize}
	\item $a'_1 = a_1 + 1$ and $a'_i = a_i$ for $i \geq 2$ if $a_1 + 1 < b_1$ (no carryover at all);
	\item $a'_1 = 0$, $a'_2 = a_2 + 1$ and $a'_i = a_i$ for $i \geq 3$ if $a_1 + 1 = b_1$ and $a_2 + 1 < b_2$ (carryover from the first position to the second);
\end{itemize}
and so on. In general, if $k$ is the smallest position such that $a_k + 1 < b_k$, one sets $a'_i = 0$ for $i<k$, $a'_k = a_k + 1$, and $a'_i = a_i$ for $i > k$. The pair $(\Lambda, \nu)$ is an adding machine.

Remark \ref{rmk:omegalimit} agrees and partly reproduces results of Buescu, Kulczycki and Stewart \cite{buescu2, buescu1}. Adapting their result to our setting, they proved that a compact invariant subset of $X$ that is transitive and Lyapunov stable is either a periodic orbit or a Cantor set in which the dynamics is topologically conjugate to an adding machine. This statement follows easily from our arguments because a transitive compact invariant set in $\varprojlim F_n$ must be the limit of an inverse sequence that at every level $F_n$ is a transitive set and hence a single periodic orbit.

We are now ready to prove the characterization of dynamical systems that admit dynamical $\epsilon$--partitions for every $\epsilon > 0$ stated in Theorem \ref{thm:intro2}.

\begin{theorem} Let $X$ be compact and totally disconnected and let $f \colon X \to X$ be continuous. The following statements are equivalent:
\begin{itemize}
	\item[({\it i}\/)] $(X, f)$ admits dynamical $\epsilon$--partitions for every $\epsilon > 0$.
	\item[({\it ii}\/)] The $\omega$--limit of every point in $X$ is either a periodic orbit or an adding machine; moreover, these are stable with respect to clopen sets.
\end{itemize}
\end{theorem}

\begin{proof} The arguments above prove $(i) \Rightarrow (ii)$. For the converse, let $\epsilon > 0$ fixed. We first claim that for every $p \in X$, its $\omega$--limit set $\omega(p)$ admits a dynamical $\epsilon$--partition. This is clear if $\omega(p)$ is a periodic orbit. If it is an adding machine, we may identify $(\omega(p),f|_{\omega(p)})$ as an inverse limit $\varprojlim P_n$ where the $P_n$ are finite sets with periodic dynamics. Denoting by $\pi_n$ the projection of $\omega(p)$ onto $P_n$, the desired dynamical $\epsilon$--partition of $\omega(p)$ is given by $\{\pi_n^{-1}(P) : P \in P_n\}$ for sufficiently large $n$.

Since $\omega(p)$ is stable with respect to clopen sets by assumption, we can use Lemma \ref{lem:extend} to extend the partition from the previous paragraph to a dynamical $\epsilon$--partition of a neighborhood of $\omega(p)$.
The stability condition forces the union of $\omega(p)$ over all $p \in X$ to be equal to $Y = \cap_{n \ge 0} f^n(X)$ so, in particular, $Y$ is closed. Then by compactness we can find a dynamical $\epsilon$--partition of a neighborhood of $Y$ and Lemma \ref{lem:trim} finally produces a dynamical $\epsilon$--partition of $X$.
\end{proof}

\section{Appendix: \v{C}ech homology and cohomology}

We collect here some information about \v{C}ech homology and cohomology that might be convenient for the reader that is not familiar with these theories. As suitable references we might suggest \cite{hatcher, spanier1}. We have not strived for generality but for simplicity; thus, many of the arguments will be quick and tailored for dimension zero. Coefficients are always taken in $\mathbb{C}$ and not displayed explicitly in the notation.

Let $X$ be a compact subset of some Euclidean space $\mathbb{R}^m$. (For the purposes of the present paper this is no restriction, since any compact, metric, totally disconnected space is $0$--dimensional and can therefore be embedded in $\mathbb{R}$ by the classical theorem of Menger and N\"obeling; see for instance \cite{hurewiczwallman1}). Taking succesively finer cubical grids of $\mathbb{R}^m$ one easily sees that $X$ has a basis of nested neighbourhoods that are compact polyhedra $(P_n)$. The inclusion maps $P_{n+1} \subset P_n$ induce homomorphisms $H_q(P_{n+1}) \to H_q(P_n)$ and $H^q(P_{n}) \to H^q(P_{n+1})$, where $H_q$ and $H^q$ denote singular homology and cohomology, respectively. The \v{C}ech homology and cohomology of $X$ (with coefficients in some group $G$) can be defined as the suitable limit of the homology and cohomology of these $P_n$ (again, with coefficients in $G$); namely, \begin{equation} \label{eq:dirlim} \check{H}_q(X) := \varprojlim\ \left\{ H_q(P_1) \leftarrow H_q(P_2) \leftarrow H_q(P_3) \leftarrow \ldots \right\} \end{equation} and \begin{equation} \label{eq:invlim} \check{H}^q(X) := \varinjlim\ \left\{ H^q(P_1) \rightarrow H^q(P_2) \rightarrow H^q(P_3) \rightarrow \ldots \right\}\end{equation} It is a theorem that these limits are a topological invariant of $X$; that is, they do not depend on how $X$ is embedded in $\mathbb{R}^n$.

Let $f \colon X \to X$ be a continuous mapping. Still thinking of $X$ as a subset of $\mathbb{R}^n$, by the Tietze extension theorem we may extend $f$ to a continuous mapping $\hat{f} \colon \mathbb{R}^m \to \mathbb{R}^m$. The compactness of $X$ then ensures that for each $n$ there exists $n'$ such that $\hat{f}(P_{n'}) \subset P_n$, and so $\hat{f}$ induces homomorphisms $\hat{f}_* \colon H_q(P_{n'}) \to H_q(P_n)$ and $\hat{f}^* : H^q(P_{n}) \to H^q(P_{n'})$. The inverse and direct limits of these are the homomorphisms $f_*$ and $f^*$ induced by $f$ in \v{C}ech homology and cohomology; they are endomorphisms of $\check{H}_q(X)$ and $\check{H}^q(X)$ respectively. Again, it is a theorem that they are independent of the particular extension $\hat{f}$.

Let us examine these definitions in dimension zero.
\smallskip

{\it \v{C}ech cohomology}. For the polyhedra $P_i$ it is well known that $H^0(P_i;\mathbb{C})$ can be identified with the set of maps from $P_i$ to $\mathbb{C}$ that are constant over each connected component of $P_i$. Under this identification the inclusion induced homomorphisms $H^0(P_i) \to H^0(P_{i+1})$ just correspond to restricting such a map $\varphi_i \colon P_i \to \mathbb{C}$ to a map $\varphi_{i+1} := \varphi_i |_{P_{i+1}} \colon P_{i+1} \to \mathbb{C}$. Pushing this forward along the direct limit \eqref{eq:dirlim} yields an element $\varphi \in \check{H}^0(X;\mathbb{C})$ which is the restriction $\varphi := \varphi_i|_X  \colon X \to \mathbb{C}$. Notice that $\varphi$ is locally constant. Conversely, suppose that $\varphi \colon X \to \mathbb{C}$ is locally constant. Then for every $c \in \mathbb{C}$ the preimage $\varphi^{-1}(c)$ is a (possibly empty) open set and the whole collection $\{\varphi^{-1}(c) : c \in \mathbb{C}\}$ is an open partition of $X$. Since $X$ is compact, this covering must have a finite subcovering, which implies that in fact $\varphi$ takes only finitely many values $c_j$ and their preimages $C_j := \varphi^{-1}(c_j)$ are actually clopen subsets of $X$. Let $N_j$ be disjoint open neighbourhoods of $C_j$ in $\mathbb{R}^m$ and choose $i$ big enough so that $P_i \subset \cup_j N_j$. Then every connected component of $P_i$ is connected in one $N_j$ and therefore intersects, at most, one $\varphi^{-1}(c_j)$. Thus there exists $\varphi_i \colon P_i \to \mathbb{C}$ that is constant on the components of $P_i$ and such that $\varphi_i |_X = \varphi$. Consequently $\varphi$ is an element of the direct limit \eqref{eq:dirlim}. Summing up, we conclude that $\check{H}^0(X;\mathbb{C})$ can be identified as the $\mathbb{C}$--vector space of locally constant mappings $\varphi \colon X \to \mathbb{C}$. Furthermore:
\begin{itemize}
	\item In the notation introduced above $\varphi = \sum_j c_j \chi_{C_j}$, where the $C_j$ are clopen subsets of $X$. Thus the set $\{\chi_U : U \text{ is a clopen subset of $X$} \}$ generates $\check{H}^0(X;\mathbb{C})$.
	\item The induced homomorphism $f^* \colon \check{H}^0(X;\mathbb{C}) \to \check{H}^0(X;\mathbb{C})$ is similarly easy to interpret in these terms: it just maps $\varphi$ to the composition $\varphi \circ f$.
\end{itemize}

Let us mention that there is a cohomology theory called Alexander--Spanier cohomology which coincides with the theory of \v{C}ech over paracompact spaces and where the interpretation of elements of $\check{H}^0(X)$ as locally constant mappings is completely straightforward. See \cite{spanierarticulo} and \cite[Section 6.4 ff.]{spanier1} for more information.
\smallskip

{\it The continuity property of \v{C}ech cohomology}. Suppose that $X$ is the intersection of a nested sequence of compacta $X_n$, all of them embedded in $\mathbb{R}^m$. Any locally constant map from $X_n$ to $\mathbb{C}$ restricts to a locally constant map from $X$ to $\mathbb{C}$, providing homomorphisms $\check{H}^0(X_n) \to \check{H}^0(X)$ which are none other than the homomorphisms induced by the inclusions $i_n \colon X_n \subset X$. Conversely, any locally constant map $\varphi$ from $X$ to $\mathbb{C}$ can be extended to a locally constant map $\hat{\varphi}$ from a neighbourhood of $X$ to $\mathbb{C}$ as argued earlier so in particular it can be extended to $X_n$ for big enough $n$. Thus every element of $\check{H}^0(X)$ belongs to the image of some $i_n^*$. All this shows that $\check{H}^0(X) = \varinjlim\ \{\check{H}^0(X_n)\}$, with the bonding homomorphisms being induced by the inclusions $X_{n+1} \subset X_n$. This (which holds in every dimension $q$) is known as the continuity property of \v{C}ech cohomology. 
\smallskip


{\it \v{C}ech homology}. The group $H_0(P_n)$ is the $\mathbb{C}$--vector space having the set of connected components of $P_n$ as a basis and therefore can be identified as the dual vector space to $H^0(P_n)$ when the latter is thought of as the set of maps from $P_n$ to $\mathbb{C}$ that are constant on the components of $P_n$, as we did before. Thus $H_q(P_n) = \Hom (H^q(P_n);\mathbb{C})$. It is easy to see from the definitions of direct and inverse limit that $\Hom(\varinjlim\ \{V_n\};\mathbb{C}) = \varprojlim\ \{\Hom( V_n ; \mathbb{C})\}$ (this essentially amounts to saying that to define a homomorphism on a monotonically increasing sequence of vector spaces one only needs to define it on each vector space and make sure that each definition extends the previous one). Applying all this to $V_n = H^0(P_n)$ we then have the identification \begin{multline*} \check{H}_0(X) =  \varprojlim\ \{ H_0(P_n) \} = \varprojlim\ \{\Hom(H^0(P_n);\mathbb{C})\} = \\ = \Hom(\varinjlim\ \{H^0(P_n)\};\mathbb{C}) = \Hom(\check{H}^0(X);\mathbb{C})\end{multline*} which exhibits \v{C}ech homology as the dual of \v{C}ech cohomology. Similar arguments show that $f_*$ is dual to $f^*$. (This reasoning is completely general and applies in any dimension $q$).

If $\lambda$ is an eigenvalue of $f_*$ then there exists $T \in \check{H}_0(X)$ which is nonzero and such that $(f_* - \lambda \id)(T) = 0$. Regarding $\check{H}_0(X)$ as the dual of $\check{H}^0(X)$ as explained above and therefore thinking of $T$ as a linear map from $\check{H}^0(X)$ to $\mathbb{C}$, this is equivalent to saying that $T \circ (f^* - \lambda \id) = 0$; that is, $\im (f^* - \lambda \id) \subset \ker(T)$. Thus $\lambda$ is an eigenvalue of $f_*$ if and only if there exists a linear map $T \colon \check{H}_0(X) \to \mathbb{C}$ that satisfies $\im (f^* - \lambda \id) \subset \ker( T)$, and this clearly happens if and only if $f^* - \lambda \id$ is not surjective.

\bibliographystyle{plain}
\bibliography{biblio}

\end{document}